\documentclass[reqno,10pt]{amsart}
\usepackage{amsmath}	
\usepackage{amssymb}	
\usepackage{amsthm}
\usepackage{amsfonts}
\usepackage{float}
\usepackage{latexsym}
\usepackage{mathrsfs}
\usepackage{tabularx}
\usepackage{pict2e}
\usepackage{nicematrix}
\usepackage{graphicx}
\usepackage[all]{xy}
\usepackage{tikz-cd}
\usepackage{tikz}
\usetikzlibrary{matrix,arrows}
\usepackage[colorlinks=true, citecolor=blue, urlcolor=blue, linkcolor=blue]{hyperref}

\theoremstyle{plain} 
\newtheorem{thm}{Theorem}[section]
\newtheorem{lem}[thm]{Lemma} 
\newtheorem{prop}[thm]{Proposition} 
\newtheorem{cor}[thm]{Corollary}

\theoremstyle{definition} 
\newtheorem{defn}[thm]{Definition}
\newtheorem{rem}[thm]{Remark} 
\newtheorem{ex}[thm]{Example}

\makeatletter
\@namedef{subjclassname@2020}{%
  \textup{2020} Mathematics Subject Classification}
\makeatother

\newcommand{\A}{\mathcal{A}}

\newcommand{\ZN}{\mathbb{Z}}

\newcommand{\PB}{\mathbb{P}}
\newcommand{\CO}{\mathcal{O}}
\newcommand{\CT}{\mathcal{T}}
\newcommand{\CN}{\mathbb{C}}

\newcommand{\DC}{\mathrm{D}^{{\rm b}}}
\newcommand{\Hom}{\mathrm{Hom}}
\newcommand{\homd}{\mathrm{hom}}
\newcommand{\Ext}{\mathrm{Ext}}
\newcommand{\ext}{\mathrm{ext}}

\newcommand{\RHom}{\mathrm{RHom}}

\newcommand{\CH}{\mathcal{H}}
\newcommand{\HH}{\mathrm{HH}}
\newcommand{\HO}{H}

\newcommand{\NHH}{\mathrm{NHH}}

\newcommand{\supp}{\mathrm{supp}}

\allowdisplaybreaks
\numberwithin{equation}{section}

\begin{document}

\title{A new phantom on a rational surface}

\author{Shihao Ma}
\address{Center for Applied Mathematics and KL-AAGDM, Tianjin University, Weijin Road 92, Tianjin 300072, P. R. China}%
\email{shma@tju.edu.cn}%

\author{Yirui Xiong}
\address{School of Sciences, Southwest Petroleum University, Chengdu 610500, P. R. China}
\email{yiruimee@gmail.com}%

\author{Song Yang}
\address{Center for Applied Mathematics and KL-AAGDM, Tianjin University, Weijin Road 92, Tianjin 300072, P. R. China}%
\email{syangmath@tju.edu.cn}%

\begin{abstract}
We construct a universal phantom subcategory on the blow-up of the complex projective plane at 11 general points. 
This phantom subcategory is the orthogonal complement of a non-full exceptional collection of line bundles of maximal length. 
It provides a new counterexample to a conjecture of Kuznetsov and to a conjecture of Orlov.
The first counterexample was constructed by Krah [Invent. Math. {\bf 235} (2024), 1009--1018].
As an application,
we construct a new co-connective DG-algebra whose derived category
is a phantom.
We also show in Appendix B that every smooth projective surface with an effective smooth anti-canonical divisor has no phantom subcategories, e.g. weak del Pezzo surfaces.
\end{abstract}

\date{\today}


\maketitle


\section{Introduction}

The purpose of this paper is to provide a new phantom subcategory of the bounded derived category of coherent sheaves on a smooth rational projective surface.
A full triangulated subcategory $ \mathcal{A}\subset \DC(X)$, where $X$ is a smooth projective variety, is called  {\it admissible} if the natural inclusion functor admits a right and a left adjoint.
An admissible subcategory of  $\DC(X)$ is called a {\it quasi-phantom} if its Hochschild homology $\HH_{\bullet}(\A)$ vanishes and its Grothendieck group $K_{0}(\A)$ is finite; moreover, if its Grothendieck group vanishes, it is called a {\it phantom}.
The existence of phantom and quasi-phantom subcategories, once considered pathological, its existences are now of great interest.
The first examples of quasi-phantom subcategories were constructed in the derived categories of some surfaces of general type (\cite{BGvBS13, AO13, GS13}), and then phantoms were discovered on the product of surfaces with quasi-phantoms (\cite{GO13}) and on a Barlow surface (\cite{BGvBKS15}); see e.g. \cite{GKMS15,CY18,KK23} for more examples.

For each smooth rational projective surface, by Orlov's blow-up formula \cite{Orl92} and mutations, its derived category always has a full exceptional collection of line bundles.
It was conjectured that there exists no phantom subcategory on smooth rational projective surfaces; see Kuznetsov \cite[Conjecture 1.10]{Kuz14} and Orlov \cite[Conjecture 3.7]{Orl20}.
Recently, Krah \cite{Kra24} constructed an exceptional collection of line bundles of maximal length on the blow-up of the complex projective plane $\mathbb{P}^2$ at $10$ general points such that its orthogonal complement is a universal phantom subcategory. 
This provides the first counterexample for these two conjectures.
In contrast, it is known that all del Pezzo surfaces do not admit phantom subcategories (see \cite{Pir23});
more recently, the blow-up of $\mathbb{P}^2$ at finite set of generic points on a smooth cubic curve  admits no phantom subcategories (see \cite{BK25}). 
In Proposition \ref{Wdl-no-phantoms}, we apply Pirozhkov's results \cite{Pir23,Pir26} to show that every smooth projective surface with an effective smooth anti-canonical divisor has no phantom subcategories following the strategy in \cite{BK25}, e.g. weak del Pezzo surfaces.
Thus, this confirms the conjecture of Borisov--Kemboi \cite[Conjecture 1.3]{BK25} for these surfaces.
Consequently, the questions have arisen as to whether it is possible to construct further phantom subcategories on smooth rational projective surfaces and classify which  surfaces admit phantom subcategories.

It is well-known that an exceptional collection  must be a numerically exceptional collection. 
In this paper, via a slightly different construction of numerically exceptional collections of maximal length, 
we can recover Krah's example and find a new universal phantom category on a smooth rational projective surface.
More precisely, suppose that  $Y$ is the blow-up of the complex projective plane $\mathbb{P}^{2}$ at $11$ general points $p_{i}$, where $1 \leq i\leq 11$.
Let $H$ be the divisor on $Y$ obtained by pulling back a hyperplane in $\PB^{2}$ and $E_{i}$ the exceptional divisor over the point $p_{i}$.
We consider the divisors 
$$
D_{i}:=-3H+\sum_{j=1}^{11}E_{j}-E_{i} \textrm{ and } F:=-10H+3\sum_{j=1}^{11}E_{j},
$$
where $1\leq i\leq 11$.

The main result of this paper is stated as follows:

\begin{thm}\label{main-thm}
There is a semi-orthogonal decomposition 
$$
\DC(Y)=\langle \A_{Y},\CO_{Y},\CO_{Y}(D_{1}),\cdots,\CO_{Y}(D_{11}),\CO_{Y}(F),\CO_{Y}(2F)\rangle,
$$
where $\A_{Y}\subset \DC(Y)$ is a non-trivial universal phantom subcategory.
\end{thm}

This theorem  provides a  new counterexample to the conjecture of Kuznetsov and to the conjecture of Orlov,  
and also to the Jordan--H\"{o}lder property for semi-orthogonal decompositions (see e.g. \cite{BGvBS13, Kuz13, Kra24, AK25} for other counterexamples of Jordan--H\"{o}lder property).
Moreover, it also provides a  new geometric counterexample  to a conjecture of Bondal--Polishchuk \cite{BP93} on the transitivity of the braid group action on full exceptional sequences in a triangulated category (see \cite{CHS25} for the first counterexample and \cite{Kra24} for the first rational surface case).

Based on Krah’s example \cite{Kra24},
Mattoo \cite[Proposition 3.3]{Mat25} constructed a strong generator $\mathcal{Q}$ on a phantom category such that $\RHom(\mathcal{Q},\mathcal{Q})$ is a co-connective DG-algebra (see \cite[Theorem 1.1]{Mat25}).
This affirmatively answers a question of Ben Antieau: ``Does there exist a co-connective DG-algebra whose derived category
is a phantom?"
As an application of Theorem \ref{main-thm}, we give further evidence for this question:   

\begin{thm}[Theorem \ref{main-corv2}]\label{main-cor}
There exists a  strong generator $\mathcal{T}$ on the phantom subcategory 
$$
\mathcal{P}_{Y}:=
\langle \CO_{Y}(-2F),\CO_{Y}(-F),\CO_{Y}(-D_{1}),\ldots,\CO_{Y}(-D_{11}),\CO_{Y} \rangle^{\perp} \subset \DC(Y)
$$ 
such that $\RHom(\mathcal{T},\mathcal{T})$ is a co-connective DG-algebra.
\end{thm}

The structure of this paper is organized as follows. 
In Section \ref{sec-2-Pre},  we recall some basic definitions and facts on semi-orthogonal decompositions, phantom subcategories, height and pseudoheight of exceptional collections and blow-up of $\PB^{2}$ at points. Section \ref{sec-3-construct} gives the construction of a numerically exceptional collection on the blow-up of $\PB^{2}$ at $11$ points in general position.
The main Theorem \ref{main-thm} is proved in Section \ref{sec-4-pf-thm1}.  
In Section \ref{sec-5-pf-thm2}, we construct a new co-connective DG-algebra whose derived category is a phantom category (Theorem \ref{main-corv2}) and consider the projections of skyscraper sheaves. 
In Appendix \ref{Appendix-HH-AY}, we obtain the dimension of the second Hochschild cohomology of the phantom category $\A_{Y}$ in Theorem \ref{main-thm}.
In Appendix \ref{No-Ph-Wdel}, we prove there exists no phantom subcategories on smooth projective surfaces with an effective smooth
anti-canonical divisor.

Right before the first version \cite{MXY25} of this paper appeared, another work by K. Kemboi, D. Krashen, T. Liu, Y. Liu, E. Mackall, S. Makarova, A. Perry, A. Robotis, and S. Venkatesh \cite{KKL+26} was announced, which also obtained the non-full exceptional collection of line bundles of Theorem \ref{main-thm} following Krah’s method. Our construction differs slightly from theirs.

\subsection*{Acknowledgments}
This work is partially supported by the National Natural Science Foundation of China (No. 12171351 and No. 12501051), and by Sichuan Science and Technology Program (No. 2025ZNSFSC0800).


\section{Preliminaries}\label{sec-2-Pre}

\subsection{Semi-orthogonal decompositions}

Let $X$ be a smooth complex projective variety and $\DC(X)$ the bounded derived category of coherent sheaves on $X$.
 
\begin{defn}
For a positive integer $l\in \ZN$,
an ordered sequence of full triangulated subcategories $\{\A_{1},\cdots,\A_{l}\}$ of $\DC(X)$ is called a {\it semi-orthogonal decomposition} of $\DC(X)$ if the following conditions hold:
\begin{enumerate}
\item[(1)] for all $A_{i}\in\A_{i}$, $A_{j} \in \A_{j}$, one has $\Hom(A_{i},A_{j})=0$ if $j<i$;
\item[(2)] for any object $T\in \DC(X)$, there exists a chain of morphisms 
$$
\xymatrix@C=0.5cm{
0=T_{l} \ar[r] & T_{l-1} \ar[r] & \cdots \ar[r] & T_{1} \ar[r] & T_{0}=T
}
$$
such that the cone $\mathrm{Cone}(T_{i}\rightarrow T_{i-1})\in \A_{i}$ for all $1\leq i\leq l$.
\end{enumerate}
Such a semi-orthogonal decomposition is denoted by 
$$
\DC(X)=\langle \A_{1},\A_{2},\cdots,\A_{n} \rangle.
$$
\end{defn}

\begin{defn}
A full triangulated subcategory $ \mathcal{A} \subset \DC(X) $ is called {\it admissible} if the inclusion functor $ \mathcal{A} \hookrightarrow \DC(X) $ admits both a right and a left adjoint. 
\end{defn}

Let $\A\subset \DC(X)$ be a full triangulated subcategory. Then the {\it left} and {\it right orthogonal complements} of $\A$
are respectively defined as follows:
$$
{}^{\perp}\A := \{ E \in \DC(X) \mid \Hom_{\DC(X)}(E, A) = 0 \text{ for all } A \in \A \}
$$
and
$$
\A^{\perp} := \{ E \in \DC(X) \mid \Hom_{\DC(X)}(A, E) = 0 \text{ for all } A \in \A \}.
$$
If $\A\subset \DC(X)$ is an admissible subcategory, then both ${}^{\perp}\A$ and  $\A^{\perp}$ are admissible subcategories and we have two semi-orthogonal decompositions 
$$
\DC(X)=\langle \A^{\perp},\A\rangle=\langle \A,{}^{\perp}\A \rangle.
$$

\begin{defn}
An object $ A\in \DC(X)$ is called {\it exceptional} if 
$\Hom(A,A)=\CN$ and $\Ext^{k}(A,A)=0$ for all $k\neq 0$.
A sequence  of exceptional objects $\{A_{1},A_{2},\cdots,A_{l}\}$ is called an {\it exceptional collection of length $l$} if $\Ext^{k}(A_{j},A_{i})=0$ for all $1\leq i<j\leq l$ and $k\in \ZN$.
\end{defn}

Suppose that $\{A_{1},A_{2},\cdots,A_{l}\}$ is an exceptional collection on $\DC(X)$. 
Then there is a semi-orthogonal decomposition
$$
\DC(X)=\langle \A_{X}, A_{1},A_{2},\cdots,A_{l} \rangle,
$$
where $\A_{X}$ is the right orthogonal decomposition. Here, for convenience,
$A_{i}$ means the smallest full triangulated subcategory $\langle A_{i} \rangle \subset \DC(X)$  containing $A_{i}$; in particular, the admissible subcategory $\langle A_{i} \rangle$ is equivalent to the bounded derived category of coherent sheaves on a point, i.e. $\langle A_{i} \rangle \cong \DC({\rm Spec} \,\CN)$.

\begin{defn}
An exceptional collection $\{A_{1},A_{2},\cdots, A_{l}\}$ is {\it full} if  the minimal full triangulated subcategory of $\DC(X)$ containing
all objects $A_{i}$ is $\DC(X)$ itself, i.e. its left or right orthogonal complement is trivial.
\end{defn}

\begin{ex}
One of the most famous examples of full exceptional collections is the sequence of line bundles $\{\CO_{\PB^{n}},\CO_{\PB^{n}}(H),\cdots, \CO_{\PB^{n}}(nH) \}$ on the $n$-dimensional complex projective space $\PB^{n}$, where $H\subset \PB^{n}$ is a hyperplane.   
\end{ex}

Given a  semi-orthogonal decomposition $\DC(X) = \langle \A_{1}, \cdots, \A_{l} \rangle$. There is a decomposition of the Grothendieck groups and Hochschild homology groups, respectively, 
$$
K_{0}(\DC(X)) \cong K_{0}(\A_{1}) \oplus \cdots \oplus K_{0}(\A_{l})
$$
and
$$
\mathrm{HH}_{\bullet}(\DC(X))
\cong 
\mathrm{HH}_{\bullet}(\A_{1}) \oplus \cdots \oplus \mathrm{HH}_{\bullet}(\A_{l}).
$$
In particular, if $\mathbb{A}:=\{A_{1},A_{2},\cdots,A_{l}\}$ is an exceptional collection on $\DC(Y)$, then 
$$
K_{0}(\DC(X)) \cong K_{0}(\A_{X})\oplus \ZN^{l}
$$
and
$$
\mathrm{HH}_{\bullet}(\DC(X)) \cong \mathrm{HH}_{\bullet}(\A_{X})\oplus \CN^{l},
$$
where $\A_{X}$ is the right or left orthogonal complement of the sequence $\mathbb{A}$.

\begin{defn}
An admissible subcategory $\A\subset \DC(X)$ is called {\it quasiphantom} if $K_{0}(\mathcal{A})$ is a finite group
and
$\mathrm{HH}_{\bullet}(\mathcal{A})= 0$. 
It is called a \textit{phantom} if, in addition, $K_{0}(\mathcal{A}) = 0$.
\end{defn}

It immediately follows from the definition, one has:

\begin{lem}\label{orth-compl-ph}
Suppose $\DC(X)$ has a full exceptional collection of length $l$.
Then the right orthogonal complement of any exceptional collection of length $l$ is a phantom category.
\end{lem}

Let $ \A \subseteq \DC(X)$ and $ \mathcal{B} \subseteq \DC(X^{\prime})$ be full triangulated subcategories. 
Then the box tensor $\A \boxtimes \mathcal{B} \subseteq \DC(X \times X^{\prime}) $ is the smallest full triangulated subcategory of $ \DC(X \times X^{\prime}) $  closed under direct summands and containing all objects of the form $ p_{X}^{*}A \otimes^{L} p_{X^{\prime}}^{*}B $ for $ A \in \A $ and $ B \in \mathcal{B}$. 
Here, $p_{X}:X\times X^{\prime}\to X$ and $p_{X^{\prime}}:X\times X^{\prime} \to X^{\prime}$ are natural projections.  

\begin{defn}
An admissible subcategory $ \A \subseteq \DC(X) $ is called a {\it universal phantom} if for all smooth projective varieties $X^{\prime}$, the category $ \A \boxtimes \DC(X^{\prime}) $ is also a phantom.
\end{defn}

We denote by $M(X)$ the Chow motive of $X$ over integral coefficients and $\mathbb{L}$ the Lefschetz motive (we refer to \cite{Ma68} for completeness). 
The Chow motive $M(X)$ is said to have {\it Lefschetz type} if $M(X)$ is isomorphic to the direct sum of $\mathbb{L}^{\otimes r}$. 
The following proposition will be used later:

\begin{prop}\label{prop:univ_phantom}
Let $\A\subseteq \DC(X)$ be a phantom category. 
If the Chow motive $M(X)$ has Lefschetz type, then $\A$ is a universal phantom category.
\end{prop}

\begin{proof}
This is a combination of Corollary 4.3 and Proposition 4.4 in \cite{GO13}.
\end{proof}

\subsection{Height and pseudoheight of exceptional collections}

To detect the non-fullness of an exceptional collection, we need the notion of height introduced by Kuznetsov
{\cite[Definition 3.2]{Kuz15}}.
Let $X$ be a smooth complex projective variety and
$\mathbb{E}:=\{ E_1, \cdots, E_n\}$ an exceptional collection on $\DC(X)$.

\begin{defn}
The {\it height} $h(\mathbb{E})$ of $\mathbb{E}$ is defined as
$$
h(\mathbb{E}):=\min \{ k\in \ZN \mid \mathrm{NHH}^{k}(\mathcal{E},\mathcal{D})=0\},
$$
where $\mathcal{D}$ is a DG-enhancement of $\DC(X)$, $\mathcal{E}$ is the DG subcategory of $\mathcal{D}$ generated by $\mathbb{E}$ and $\mathrm{NHH}^{k}(\mathcal{E},\mathcal{D})$ is the normal Hochschild cohomology of $\mathcal{E}$ in $\mathcal{D}$ as a certain DG-module.
\end{defn}

We will use the height to detect the non-fullness of an exceptional collection.

\begin{lem}[{\cite[Proposition 6.1]{Kuz15}}]\label{not-full-criterion}
If the height $h(\mathbb{E})>0$, then the exceptional collection $\mathbb{E}$ is not full.
\end{lem}

In practice, we use 
the so-called pseudoheight of an exceptional collection.  
 
\begin{defn}\label{pseudoheight-defn} 
 The {\it pseudoheight} $\mathrm{ph}(\mathbb{E})$ of $\mathbb{E}$ is defined to be
$$
\min_{1 \leq a_0 < \cdots < a_p \leq n} \left[e(E_{a_0}, E_{a_1}) + \cdots + e(E_{a_{p-1}}, E_{a_p}) + e(E_{a_p}, E_{a_0} \otimes \omega_{X}^{-1}) - p \right]+\dim X,
$$
where  $e(F,G):=\inf\{k\in \ZN \mid \Ext^{k}(F,G)\neq 0\}$ is the {\it relative height} of  $F$ and $G$ in $\DC(X)$ and $\omega_{X}$ is the canonical sheaf of $X$. 
\end{defn}

By \cite[Lemma 4.5]{Kuz15}, the height $h(\mathbb{E})\geq \mathrm{ph}(\mathbb{E})$.
Sometimes, the height and the pseudoheight are equal;
for example, the following:

\begin{lem}[{\cite[Proposition 4.7]{Kuz15}}]\label{h=ph}
If $\mathrm{ph}(\mathbb{E})=e(E_{i},E_{i}\otimes\omega_{X}^{-1})+\dim Y$ for some $1\leq i\leq n$, then $h(\mathbb{E})= \mathrm{ph}(\mathbb{E})$.
\end{lem}


\subsection{Blow-up of $\PB^{2}$ at points}

Let $\pi: Y \rightarrow \PB^{2}$ be the blow-up of the complex projective plane $\PB^{2}$ in $n$ closed points $p_{i}$, where $1\leq i\leq n$.
We use $E_{i}:=\pi^{-1}(p_{i})\subset Y $ to denote the $(-1)$-curve over the points $p_{i}$.
The Picard group of $Y$ is 
$$
\mathrm{Pic}(Y)=\ZN H\oplus \ZN E_{i}\oplus \cdots \oplus \ZN E_{n},
$$
where $H$ is the divisor class obtained by pulling back the class of a hyperplane in $\PB^{2}$.
The intersection numbers 
$$
H^{2}=1,\, E_{i}^{2}=-1, \, H\cdot E_{i}=0\; \textrm{ and } E_{i}\cdot E_{j}=0 \; \textrm{for } i\neq j.
$$
The canonical class $K_{Y}=-3H+\sum_{i=1}^{n}E_{i}$ with self-intersection $K_{Y}^{2}=9-n$.
For any  divisor $D$ on $Y$, it  can be uniquely written as a combination
$$
D=dH-\sum_{i=1}^{n}m_{i}E_{i},
$$
where $d$ and all $m_{i}$ are integers.

\begin{lem}\label{SHGH-conj-lemma}
Suppose that $Y$ is the blow-up of $\PB^{2}$ at $n$ general points.
If $d>0$, $d\geq m_{1}+m_{2}+m_{3}$, $d\geq m_{1}\geq m_{2}\geq \cdots \geq m_{n}\geq 0$ and $m_{i}\leq 11$, then
the divisor $D$ satisfies
$$
h^{0}(\CO_{Y}(D))=\max\{0,\chi(\CO_{Y}(D))\}.
$$
\end{lem}

\begin{proof}
Since $m_{i}\leq 11$, by \cite[Theorem 34]{DJ07},  we have
$$
h^{0}(\CO_{Y}(D))=\max\{0,\chi(\CO_{Y}(D))\}
$$ 
or there exists a $(-1)$-curve $C \subset Y$ such that $C_{\cdot} D\leq -2$; 
namely, the SHGH conjecture holds for $D$.
Note that $d\geq m_{1}+m_{2}+m_{3}$ and $d\geq m_{1}\geq m_{2}\geq \cdots \geq m_{n}\geq 0$.
Suppose $C$ is a $(-1)$-curve. 
By \cite[Proposition 1.4]{CM11} or \cite[Lemma 3.2]{Kra24}, 
we have $C_{\cdot} D\geq 0$.
This concludes the proof.
\end{proof}


\section{Construction of numerically exceptional collections}\label{sec-3-construct}

Let $Y$ be a smooth complex projective surface.
For any object $A,B\in \DC(Y)$,
the Euler characteristic
$$
\chi(A,B):=\sum_{i} (-1)^{i} \dim \Ext^{i}(A,B).
$$

\begin{defn}
A sequence $\{A_{0},A_{1},\cdots,A_{l}\}$ of objects in $\DC(Y)$  is called a {\it numerically exceptional collection} if $\chi(A_{i},A_{i})=1$ for all $0\leq i\leq l$ and $\chi(A_{j},A_{i})=0$ for all $0 \leq i<j\leq l$.
Moreover, a numerically exceptional collection is said to be {\it of maximal length} if it spans the numerical Grothendieck group of $Y$.
Here, the numerical Grothendieck group is the quotient of the  Grothendieck group by the
kernel of the Euler form $\chi$.
\end{defn}

We are mainly interested in numerically exceptional collections of line bundles.
A sequence of line bundles $\{A_{0},A_{1}, \cdots, A_{l}\}$ is a numerically exceptional collection if and only if $\chi(\CO_{Y})=1$ and $\chi(A_{i}\otimes A_{j}^{-1})=0$ for all $ i<j$.
For any  divisor $D$ on $Y$, the Euler characteristic is given by the Riemann--Roch formula
\begin{equation*}\label{RR-formula}
\chi(\CO_{Y}(D))=\frac{1}{2}D\cdot (D-K_{Y})+\chi(\CO_{Y}).
\end{equation*}
In particular, if $\chi(\CO_{Y})=1$, then $\chi(\CO_{Y}(D))=0$ if and only if 
\begin{equation}\label{partial-RR-formula}
D\cdot (D-K_{Y})=-2.
\end{equation}

Now suppose that $Y$ is the blow-up of  $\PB^{2}$ in $n$ closed points $p_{i}$, where $n\geq 10$.
By Orlov's blow-up formula \cite{Orl92} and mutations,  
there exists a full exceptional collection of line bundles of length $n+3$
\begin{equation}\label{EFC-on-blowupP2}
\DC(Y)=\langle \CO_{Y}, \CO_{Y}(E_{1}),\cdots,\CO_{Y}(E_{n}), \CO_{Y}(H), \CO_{Y}(2H) \rangle. 
\end{equation}
We consider the  divisors 
$$
D_{i}:=aK_{Y}-E_{i}\, \textrm{ and} \; F:=bK_{Y}-H,
$$
where $1\leq i\leq n$ and $a,b\in \ZN$.
Since $\chi(\CO_{Y})=1$ and $\chi(\CO_{Y}(E_{i}-E_{j}))=0$ for all $i\neq j$,
the sequence of line bundles
$$
\{\CO_{Y},\CO_{Y}(D_{1}),\cdots,\CO_{Y}(D_{n}),\CO_{Y}(F),\CO_{Y}(2F)\}
$$
is a numerically exceptional collection if and only if
$\chi(\CO_{Y}(-D_{i}))=0$, 
$\chi(\CO_{Y}(-F))=0$,
$\chi(\CO_{Y}(-2F))=0$,
$\chi(\CO_{Y}(D_{i}-F))=0$ and
$\chi(\CO_{Y}(D_{i}-2F))=0$ for all $1\leq i \leq n$.
Equivalently,
by \eqref{partial-RR-formula},
the following equations hold:
\begin{align}
a(a+1)(9-n)+2a&=-2,\label{num-exceptional-RR-equ-1} \\
b(b+1)(9-n)+6b&=-6, \label{num-exceptional-RR-equ-2}\\
b(2b+1)(9-n)+12b&=-6,\label{num-exceptional-RR-equ-3}\\
(a-b)(a-b-1)(9-n)-4(a-b)&=-4, \label{num-exceptional-RR-equ-4}\\
(a-2b)(a-2b-1)(9-n)-10(a-2b)&=-10. \label{num-exceptional-RR-equ-5}
\end{align}
By \eqref{num-exceptional-RR-equ-2} and \eqref{num-exceptional-RR-equ-3}, we get
\begin{equation}\label{important-equ-6}
    b(9-n)=-6.
\end{equation}
It follows that $n\in \{10,11,12,15\}$.
If $n=12$, by \eqref{num-exceptional-RR-equ-1} and \eqref{important-equ-6}, 
we obtain $a=-1$ and $b=2$; this contradicts \eqref{num-exceptional-RR-equ-4}.
If $n=15$, then by \eqref{num-exceptional-RR-equ-1} and \eqref{important-equ-6}, we have $a=-1$ and $b=1$, a contradiction to \eqref{num-exceptional-RR-equ-4}.
Hence, there are only two cases satisfying \eqref{num-exceptional-RR-equ-1}-\eqref{num-exceptional-RR-equ-5}:
\begin{enumerate}
\item[(1)] If $n=10$, then $a=2$ and $b=6$. 
In fact, this case is the example of Krah \cite{Kra24}.
\item[(2)] If $n=11$, then $a=1$ and $b=3$. As far as we know, this case is new.
\end{enumerate}

In summary, we obtain a new numerically exceptional collection of maximal length on the blow-up $Y$ of  $\PB^{2}$ in $11$ closed points.
Consider the divisors
\begin{equation}\label{main-construct-divisors}
D_{i}:=K_{Y}-E_{i} \textrm{ and } F:=3K_{Y}-H,
\end{equation}
where $1\leq i\leq 11$ and $K_{Y}=-3H+\sum_{i=1}^{11}E_{i}$.
By \eqref{EFC-on-blowupP2}, the Grothendieck  group $K_{0}(Y)\cong \ZN^{14}$.
Thus, we obtain:

\begin{prop}\label{main-num-exceptional-coll}
The sequence
\begin{equation*}
\{\CO_{Y},\CO_{Y}(D_{1}),\cdots,\CO_{Y}(D_{11}),\CO_{Y}(F),\CO_{Y}(2F)\}
\end{equation*}
is a numerically exceptional collection of line bundles of maximal length.
\end{prop}

\begin{rem}
If we only require $n\geq 3$, then there exist two more solutions satisfying \eqref{num-exceptional-RR-equ-1}-\eqref{num-exceptional-RR-equ-5}:
\begin{enumerate}
\item[(1)] If $n=7$, we have the following numerically exceptional collection of line bundles of maximal length
\begin{equation}\label{num-exceptional-coll-n=7}
\{\CO_{Y},\CO_{Y}(D_{1}),\ldots,\CO_{Y}(D_{7}),\CO_{Y}(F),\CO_{Y}(2F)\},
\end{equation}
where $D_{i}=-K_{Y}-E_{i}$ for $1\leq i\leq 7$ and $F=-3K_{Y}-H$.
\item[(2)] If $n=8$, we have the following numerically exceptional collection of line bundles of maximal length 
\begin{equation}\label{num-exceptional-coll-n=8}
\{ \CO_{Y},\CO_{Y}(D_{1}),\ldots,\CO_{Y}(D_{8}),\CO_{Y}(F),\CO_{Y}(2F)\},
\end{equation}
where $D_{i}=-2K_{Y}-E_{i}$ for $1\leq i\leq 8$ and $F=-6K_{Y}-H$.
\end{enumerate}
In the above two cases, when $Y$ is a del Pezzo surface, two sequences \eqref{num-exceptional-coll-n=7} and \eqref{num-exceptional-coll-n=8} are full exceptional collections (see \cite[Theorem 1.4]{EL16}). 
Hence, a natural question arises:
{\it are the sequences \eqref{num-exceptional-coll-n=7} and \eqref{num-exceptional-coll-n=8} non-full exceptional collections for $Y$ being weak del Pezzo but not del Pezzo}? Notably, it was conjectured that there are no phantoms  on these weak del Pezzo surfaces (see \cite[Conjecture 1.3]{BK25}).
In Appendix \ref{No-Ph-Wdel}, we confirm Borisov--Kemboi's conjecture for weak del Pezzo surfaces.
Thus,  both the sequences \eqref{num-exceptional-coll-n=7} and \eqref{num-exceptional-coll-n=8} cannot be non-full exceptional collections.
\end{rem}


\section{Proof of Theorem \ref{main-thm}}\label{sec-4-pf-thm1}

Let $Y$ be the blow-up of $\PB^{2}$ along $11$ general points.
By Proposition \ref{main-num-exceptional-coll}, there is a numerically exceptional collection of line bundles of maximal length
\begin{equation}\label{num-exceptional-collection}
\{\CO_{Y},\CO_{Y}(D_{1}), \cdots,\CO_{Y}(D_{11}),\CO_{Y}(F),\CO_{Y}(2F)\},
\end{equation}
where the divisors $D_{i}$ and $F$ are defined as in \eqref{main-construct-divisors}.

\begin{lem}\label{excep-coll}
The sequence \eqref{num-exceptional-collection} is an exceptional collection.
\end{lem}

\begin{proof}
Based on Proposition \ref{main-num-exceptional-coll}, it is sufficient to verify that the vanishing of $\Hom$-spaces and $\Ext^{2}$-spaces.
By Serre duality, we have
\begin{align*}
\Ext^{2}(\CO_{Y}(D_{i}),\CO_{Y})&  \cong H^{0}(\CO_{Y}(2K_{Y}-E_{i})),\\
\Ext^{2}(\CO_{Y}(F),\CO_{Y})& \cong H^{0}(\CO_{Y}(4K_{Y}-H)),\\
\Ext^{2}(\CO_{Y}(2F),\CO_{Y})& \cong H^{0}(\CO_{Y}(7K_{Y}-2H)),\\
\Ext^{2}(\CO_{Y}(2F),\CO_{Y}(F))& \cong H^{0}(\CO_{Y}(4K_{Y}-H)),\\
\Ext^{2}(\CO_{Y}(F),\CO_{Y}(D_{i}))&\cong H^{0}(\CO_{Y}(3K_{Y}+E_{i}-H)),\\
\Ext^{2}(\CO_{Y}(2F),\CO_{Y}(D_{i}))& \cong H^{0}(\CO_{Y}(6K_{Y}+E_{i}-2H)).
\end{align*}
All the above divisors intersect $H$ negatively, so  all $\Ext^{2}$-spaces vanish.
Since $\{\CO_{Y}(E_{i}),\CO_{Y}(E_{j})\}$ is an exceptional collection, we have
$$
\Ext^{k}(\CO_{Y}(D_{j}),\CO_{Y}(D_{i}))=\Ext^{k}(\CO_{Y}(E_{i}),\CO_{Y}(E_{j}))=0
$$ 
for $k=0,2$.
The remaining cases are 
\begin{align*}
\Hom(\CO_{Y}(D_{i}),\CO_{Y})&=H^{0}(\CO_{Y}(3H-\sum_{i=1}^{11}E_{i}+E_{i})),\\
\Hom(\CO_{Y}(F),\CO_{Y})&=H^{0}(\CO_{Y}(10H-3\sum_{i=1}^{11}E_{i})),\\
\Hom(\CO_{Y}(2F),\CO_{Y})&=H^{0}(\CO_{Y}(20H-6\sum_{i=1}^{11}E_{i})),\\
\Hom(\CO_{Y}(2F),\CO_{Y}(F))&=H^{0}(\CO_{Y}(10H-3\sum_{i=1}^{11}E_{i})),\\
\Hom(\CO_{Y}(F),\CO_{Y}(D_{i}))&=H^{0}(\CO_{Y}(7H-2\sum_{i=1}^{11}E_{i}-E_{i})),\\
\Hom(\CO_{Y}(2F),\CO_{Y}(D_{i}))&=H^{0}(\CO_{Y}(17H-5\sum_{i=1}^{11}E_{i}-E_{i})).\\
\end{align*}
Note that the divisors in the remaining cases satisfy the conditions in Lemma \ref{SHGH-conj-lemma}. Thus, all $\Hom$-spaces vanish.
This completes the proof of Lemma \ref{excep-coll}.
\end{proof}

\begin{lem}\label{excep-coll3}
The height of the exceptional collection  \eqref{num-exceptional-collection} is $3$.
In particular, the sequence \eqref{num-exceptional-collection} is not full.
\end{lem}

\begin{proof}
Set $L_{0}:=\CO_{Y}$,
$L_{1}:=\CO_{Y}(D_{1}),\cdots,L_{11}:=\CO_{Y}(D_{11})$,
$L_{12}:=\CO_{Y}(F)$,
$L_{13}:=\CO_{Y}(2F)$.
First, we calculate the relative height $e(L_{i},L_{j})$ for all $i<j$. 
We have three cases.
(1) The relative height $e(L_{0},L_{i})=2$ for $1\leq i\leq 11$. In fact, since $D_{i}\cdot H<0$, so  $\homd(L_{0},L_{i})=0$.
Since $E_{i}$ is a $(-1)$-curve,
by Serre duality, 
we obtain 
$\ext^{2}(L_{0},L_{i})=1$.
By Riemann--Roch formula, we have $\ext^{1}(L_{0},L_{i})=1-\chi(L_{0},L_{i})=0$.
(2) The relative height $e(L_{i},L_{j})=\infty$ for $1\leq i<j\leq 11$, since $\{E_{i}$, $E_{j}\}$ is an exceptional pair. 
(3) The relative height $e(L_{i},L_{j})=2$,
for $(L_{i},L_{j})\neq (\CO_{Y}(D_{i}),\CO_{Y}(D_{j}))$ and $(L_{i},L_{j})\neq (\CO_{Y},\CO_{Y}(D_{j}))$.
Since $c_{1}(L_{j}\otimes L_{i}^{-1})\cdot H<0$, we get $\Hom(L_{i},L_{j})=0$.
By Riemann--Roch formula, we derive $\chi(L_{i},L_{j})>0$.
By Serre duality and Lemma \ref{SHGH-conj-lemma}, we deduce
$
\ext^{2}(L_{i},L_{j})=\chi(L_{i},L_{j}).
$
It follows that $\ext^{1}(L_{i},L_{j})=0$.
Thus, $e(L_{i},L_{j})=2$.

Next, we compute the relative height $e(L_{j},L_{i}\otimes\omega_{Y}^{-1})$ for $i<j$. 
By Riemann--Roch formula, we have
$\chi(L_{j},L_{i}\otimes\omega_{Y}^{-1})<0$ for $i<j$.
Since $c_{1}(L_{j}\otimes L_{i}^{-1}\otimes \omega_{Y}^{\otimes 2}){\cdot}H<0$, by Serre duality,  we obtain $\Ext^{2}(L_{j},L_{i}\otimes \omega_{Y}^{-1})\cong H^{0}(L_{j}\otimes L_{i}^{-1}\otimes \omega_{Y}^{\otimes 2})=0$.
We discuss two cases:
(i) The relative height $e(L_{j},L_{i}\otimes\omega_{Y}^{-1})=1$ for $1\leq i<j\leq 11$. 
We have the short exact sequence
\begin{equation}\label{excep-coll-not-full-SES-1}
\xymatrix@C=0.5cm{
0 \ar[r] & L_{i}\otimes L_{j}^{-1}\otimes \omega_{Y}^{-1} \ar[r] & L_{j}^{-1} \ar[r] & \CO_{E_{i}}(-K_{Y}+E_{j}) \ar[r] & 0.
} 
\end{equation}
Since $H^{0}(L_{j}^{-1})=H^{0}(\CO_{Y}(-K_{Y}+E_{j}))=0$, so \eqref{excep-coll-not-full-SES-1} implies $\Hom(L_{j},L_{i}\otimes \omega_{Y}^{-1} )=0$. 
Hence, $\ext^{1}(L_{j},L_{i}\otimes\omega_{Y}^{-1})=-\chi(L_{j},L_{i}\otimes\omega_{Y}^{-1})>0$.
(ii) The relative height $e(L_{j},L_{i}\otimes \omega_{Y}^{-1})=1$ for $(L_{i},L_{j})\neq (\CO_{Y}(D_{i}),\CO_{Y}(D_{j}))$.
It follows from Lemma \ref{SHGH-conj-lemma} that $\Hom(L_{j},L_{i}\otimes \omega_{Y}^{-1})\cong H^{0}(L_{j}^{-1}\otimes L_{i}\otimes \omega_{Y}^{-1})=0$.
Since $c_{1}(L_{j}\otimes L_{i}^{-1}\otimes \omega_{Y}^{\otimes 2}){\cdot}H<0$, by Serre duality,  we obtain $\Ext^{2}(L_{j},L_{i}\otimes \omega_{Y}^{-1})\cong H^{0}(L_{j}\otimes L_{i}^{-1}\otimes \omega_{Y}^{\otimes 2})=0$.
Thus, we derive $\ext^{1}(L_{j},L_{i}\otimes\omega_{Y}^{-1})=-\chi(L_{j},L_{i}\otimes\omega_{Y}^{-1})>0$.

Finally, following the same arguments as (ii), we have
$e(L_{i},L_{i}\otimes\omega_{Y}^{-1})=1$ for all $i$.
By definition, the pseudoheight of \eqref{num-exceptional-collection} is $3$.
Since the relative height $e(L_{i},L_{i}\otimes \omega_{Y}^{-1}[-2])=3$, 
it follows from Lemma \ref{h=ph} that the height of \eqref{num-exceptional-collection} is $3$.
In particular, by Lemma \ref{not-full-criterion}, the sequence \eqref{num-exceptional-collection} is not full.
\end{proof}

\begin{thm}\label{our-universal-phantom}
The right orthogonal complement of the sequence \eqref{num-exceptional-collection},
$$
\A_{Y}:=\langle \CO_{Y},\CO_{Y}(D_{1}), \cdots,\CO_{Y}(D_{11}),\CO_{Y}(F),\CO_{Y}(2F) \rangle^{\perp} \subset \DC(Y)
$$
is a universal phantom subcategory.
\end{thm}

\begin{proof}
By Lemma \ref{orth-compl-ph}, the admissible subcategory $\A_{Y}$ is a phantom category.
We use Manin's blow-up formula of Chow motive \cite[Section 9]{Ma68} to get $M(Y) = \mathbf{1}\oplus \mathbb {L}^{\oplus 12} \oplus \mathbb{L}^{\otimes 2}$. 
Therefore, the Chow motive $M(Y)$ has Lefschetz type. The result follows from Proposition \ref{prop:univ_phantom}.
\end{proof}

\begin{rem}
(1) The height of the non-full exceptional collection constructed in \cite{Kra24} is $4$; in particular, the formal deformation space of Krah's phantom category is isomorphic to that of the derived category of the base rational surface (see \cite[Remark 5.5]{Kra24}). 
However, the height of our non-full exceptional collection \ref{num-exceptional-collection} is $3$.
It is unknown whether the formal deformation spaces of $\DC(Y)$ and $\A_{Y}$ are isomorphic.

(2) It was proved that the second Hochschild cohomology $\mathrm{HH}^{2}(\A_{Y})$ of dimension is at least $\dim \mathrm{HH}^{2}(Y)=14$ (see \cite[Proposition 3.5]{KKL+26}).
In particular, the phantom category $\A_{Y}$ is not equivalent to Krah's phantom category (see \cite[Remark 3.6]{KKL+26}).  Moreover, we will show that $\dim \mathrm{HH}^{2}(\A_{Y})=27$ (see Proposition \ref{normal-Hoch-cohomology-Y}).

(3) Let $X$ be the Hilbert scheme of $n$ points on $Y$. 
According to \cite[Theorem 3.4]{Kos22}, 
for each $1 \leq i\leq n$, the symmetric product ${\rm Sym}^{i}(\A_{Y})\subset \DC(X)$ is a phantom admissible category. 
By \cite[Theorem 4.1]{Tot20}, the Chow motive $M(X)$ is of Lefschetz type.
Then, by Proposition \ref{prop:univ_phantom}, 
${\rm Sym}^{i}(\A_{Y})$ is a universal phantom category for $1\leq i\leq n$.
\end{rem}

 
\section{A co-connective DG-algebra and projected skyscraper sheaves}\label{sec-5-pf-thm2}

In this section, we construct a new co-connective DG-algebra whose derived category is a phantom category and obtain some non-trivial objects in the phantom category by projecting the skyscraper sheaves.

To begin with, we set the ordered sequence
\begin{equation}\label{new-equ-coll}
\{\mathcal{E}_{1},\mathcal{E}_{2},\cdots,\mathcal{E}_{14}\}:=\{\CO_{Y}(-2F),\CO_{Y}(-F),\CO_{Y}(-D_{1}),\ldots,\CO_{Y}(-D_{11}),\CO_{Y}\}
\end{equation}
whose objects are the duals of those in \eqref{num-exceptional-collection}.
By Lemmas \ref{excep-coll} and \ref{excep-coll3}, 
we have the following two lemmas:

\begin{lem}\label{strong-generator-lem-1}
The sequence \eqref{new-equ-coll} is a non-full exceptional collection.
\end{lem}

\begin{lem}\label{strong-generator-lem-2}
$\RHom(\mathcal{E}_{i},\mathcal{E}_{j})=\CN^{\chi(\mathcal{E}_{j}\otimes\mathcal{E}_{i}^{\vee})}[-2]$ for $1\leq i<j\leq 14$.
\end{lem}

We denote $\mathcal{P}_{Y}:=\langle\mathcal{E}_{1},\ldots,\mathcal{E}_{14}\rangle^{\perp}$ the  right orthogonal complement and $\iota:\mathcal{P}_{Y} \hookrightarrow \DC(Y)$ the natural inclusion functor.
Based on Lemma \ref{strong-generator-lem-1}, using the same argument as the proof of Theorem \ref{our-universal-phantom}, we can obtain:

\begin{prop}\label{P_Y-phantom}
 The admissible subcategory $\mathcal{P}_{Y}$ is a phantom category.   
\end{prop}

Moreover, there is an exact equivalence between two triangulated categories
$$
\Phi\circ \iota: \mathcal{P}_{Y}\rightarrow \A_{Y},
$$
where $\Phi:=R\CH om(-,\omega_{Y}):  \mathrm{D}^{b}(Y) \rightarrow  \mathrm{D}^{b}(Y)$.
Since $\A_{Y}$ is a phantom, this also implies that $\mathcal{P}_{Y}$ is a phantom (see Proposition \ref{P_Y-phantom}).

\subsection{Proof of Theorem \ref{main-cor}}
This subsection proves Theorem \ref{main-cor} by
using the strategy as that of \cite[Theorem 1.1]{Mat25}.
We denote 
$$
\{\mathcal{F}_{1},\ldots, \mathcal{F}_{14} \}:=\{ \CO_{Y},\CO_{Y}(E_{1}),\ldots,\CO_{Y}(E_{11}),\CO_{Y}(H),\CO_{Y}(2H)\}
$$
the full exceptional collection \eqref{EFC-on-blowupP2} and the object
$$
\mathcal{T}:=\bigoplus_{i=2}^{14}\iota^{\ast}\mathcal{F}_{i} \in \mathcal{P}_{Y},
$$
where $\iota^{\ast}:\DC(Y)\rightarrow \mathcal{P}_{Y}$ is the left adjoint functor of the natural inclusion functor $\iota$, 
i.e. $\iota^{\ast}$ is the left mutation functor $\mathrm{L}_{\langle\mathcal{E}_{1},\mathcal{E}_{2},\cdots,\mathcal{E}_{14} \rangle}$ through $\langle\mathcal{E}_{1},\mathcal{E}_{2},\cdots,\mathcal{E}_{14} \rangle$.
In particular, we have $\iota^{\ast} \CO_{Y}=0$.

\begin{thm}[Theorem \ref{main-cor}]\label{main-corv2}
The object $\mathcal{T}$ is a strong generator of $\mathcal{P}_{Y}$ and $A:=\RHom(\mathcal{T},\mathcal{T})$ is a co-connective DG-algebra.
In particular, the derived category of $A$ is a phantom.
\end{thm}

\begin{proof}
Since $\{\mathcal{F}_{1},\cdots,\mathcal{F}_{14}\}$ is a full exceptional collection,
every object $K\in \DC(Y)$ can be written in terms of direct sums, cones, and summands of the $\mathcal{F}_{i}$'s.
Since $\iota^{\ast}\circ \iota$ is isomorphic to $\mathrm{id}_{\mathcal{P}_{Y}}$, every element of the phantom $\mathcal{P}_{Y}$ can be written as $\iota^{\ast}K$ for $K\in \DC(Y) $, and all these operations commute with $\iota^{\ast}$. 
Thus, $\mathcal{T}$ is a strong generator of $\mathcal{P}_{Y}$.

Next, we show that $\RHom(\mathcal{T},\mathcal{T})$ is a co-connective DG-algebra, i.e. for any $2 \leq i,j\leq 14$, Ext group $\Ext^{m}(\iota^{\ast}\mathcal{F}_{i},\iota^{\ast}\mathcal{F}_{j})=0$ for $m<0$.
For this purpose, we consider the spectral sequence in \cite[Corollary 2.18]{Mat25}: for $K^{\prime},K\in \DC(Y)$, 
there is a spectral sequence converging to $\textnormal{Hom}^{p+q}(i^{*}K',i^{*}K)$ with $E_{1}^{-p-1,q}$ equal to
\begin{equation}\label{proj-spectral-sq}
\bigoplus_{\substack{1\leq a_{0}<...<a_{p}\leq 14,\\ k_{0}+...+k_{p}+k=q}}
\begin{aligned}[t]
&\textnormal{Hom}^{k}(K',\mathcal{E}_{a_{0}})\otimes\textnormal{Hom}^{k_{0}}(\mathcal{E}_{a_{0}},\mathcal{E}_{a_{1}})\otimes...\otimes\textnormal{Hom}^{k_{p-1}}(\mathcal{E}_{a_{p-1}},\mathcal{E}_{a_{p}})\\
& \qquad  \qquad \qquad  \qquad \otimes\textnormal{Hom}^{k_{p}}(\mathcal{E}_{a_{p}},K)
\end{aligned}
\end{equation}
for $p\geq0$ and $E_{1}^{0,q}=\textnormal{Hom}^{q}(K',K)$, where the differential $d_{1}$ given by signed composition.
By Lemma \ref{strong-generator-lem-2}, we know that $\Ext^{m}(\mathcal{E}_{i},\mathcal{E}_{j})$ is non-zero only for $m=2$.
Hence, $E_{1}^{p,q}$ is $0$ for $p > 0$, for $p = 0, q < 0$, for $p = -1, q < 0$, and for $p < -1, q < -2p-2$. 
The only potentially non-zero term with negative total degree is $E_{1}^{-1,0}$, and  the only non-zero differential into or out of $E_{1}^{-1,0}$ is $d_{1}: E_{1}^{-1,0} \to E_{1}^{0,0}$.
This yields $\Ext^{m}(\iota^{\ast}\mathcal{F}_{i},\iota^{\ast}\mathcal{F}_{j})=0$ for $m\leq -2$ and 
$$
\Ext^{-1}(\iota^{\ast}\mathcal{F}_{i},\iota^{\ast}\mathcal{F}_{j}) = \ker\left( \bigoplus_{k=1}^{14} \Hom(\mathcal{F}_{i}, \mathcal{E}_{k}) \otimes \Hom(\mathcal{E}_{k},\mathcal{F}_{j}) \xrightarrow{\varphi_{k}} \Hom(\mathcal{F}_{i},\mathcal{F}_{j}) \right),
$$
where $\varphi_{k}$ is the natural composition map.
To show $\Ext^{-1}(\iota^{\ast}\mathcal{F}_{i},\iota^{\ast}\mathcal{F}_{j})=0$ for all $2\leq i,j\leq 14$, it is sufficient to verify that the kernel of $\varphi_{k}$ is trivial for all $k$.
Since $2\leq i\leq 14$, so $\mathcal{F}_{i}\neq \CO_{Y}$, by direct computations, we have $\Hom(\mathcal{F}_{i},\mathcal{E}_{k})\otimes \Hom(\mathcal{E}_{k},\mathcal{F}_{j})=0$.
Thus,  $\Ext^{-1}(\iota^{\ast}\mathcal{F}_{i},\iota^{\ast}\mathcal{F}_{j})$ is trivial.

Finally, since $\mathcal{T}$ is a strong generator of $\mathcal{P}_{Y}$,
the phantom subcategory $\mathcal{P}_{Y}$ is  equivalent to the derived category of the DG-algebra $\RHom(\mathcal{T},\mathcal{T})$.
This completes the proof.
\end{proof}

\begin{rem}
As mentioned above, the phantom category $\A_{Y}$ is not equivalent to Krah's phantom category. 
In particular, the co-connective DG-algebra $\RHom(\mathcal{T},\mathcal{T})$ in Theorem \ref{main-corv2} is not derived equivalent to  the co-connective DG-algebra obtained by Mattoo in \cite[Theorem 1.1]{Mat25}.
\end{rem}

\subsection{Projections of skyscraper sheaves}

The problem of detecting non-trivial objects in a phantom subcategory is of significant importance. 
In \cite[Proposition 4.2]{Mat25}, Mattoo proved that the projections of skyscraper sheaves in a phantom subcategory are non-trivial. 
This gives some non-trivial objects in Krah's phantom subcategory. 
Likewise, we have the following:

\begin{prop}\label{RHom-projected-skyscraper-sheaves}
Let $\CO_{y}$ be the skyscraper sheaf at a closed point $y\in Y$.
Then
$$
\RHom(\iota^{\ast}\CO_{y},\iota^{\ast}\CO_{y})=\CN[0]\oplus \CN^{15}[-1]\oplus \CN^{100}[-2]\oplus \CN^{152}[-3]\oplus \CN^{66}[-4].
$$
In particular, the object $\Phi\circ \iota(\iota^{\ast} \CO_{y})\in \A_{Y}$ is non-trivial.
\end{prop}

\begin{proof} 
By Riemann--Roch formula, the Euler characteristics
$$
\chi(\CO(-D_{i}),\CO)=1, \chi(\CO(-F),\CO(-D_{i}))=2, \chi(\CO(-2F),\CO(-D_{i}))=5, 
$$
$$
\chi(\CO(-F),\CO)=3 \,\; \textrm{and}\,\; \chi(\CO(-2F),\CO)=6,
$$
where $1\leq i\leq 11$.

Note that $\RHom(\CO_{y},\CO_{y})=\CN[0]\oplus\CN^{2}[-1]\oplus\CN[-2]$ for any closed point $y\in Y$.
Let $\{\mathcal{E}_{1},\mathcal{E}_{2},\cdots,\mathcal{E}_{14}\}$ be the exceptional collection as in \eqref{new-equ-coll}.
By the spectral sequence \eqref{proj-spectral-sq}, we get $E_{1}^{0,q}$ is non-zero if and only if $q=0,1,2$.
Next, we discuss the case $p\geq 0$.
Since all $\mathcal{E}_{i}$ are line bundles, so we get $\RHom(\mathcal{E}_{i},\CO_{y})=\CN[0]$ and $\RHom(\CO_{y},\mathcal{E}_{i})=\CN[-2]$.
According to \eqref{proj-spectral-sq}, we know that $E_{1}^{-1-p,q}$ is equal to
\begin{equation}\label{skyscraper-equ-1}
\bigoplus_{\substack{1\leq a_{0}<...<a_{p}\leq 14,\\ k_{0}+...+k_{p}+k=q}}
\begin{aligned}[t]
& \Ext^{k}(\CO_{y},\mathcal{E}_{a_{0}})\otimes\Ext^{k_{0}}(\mathcal{E}_{a_{0}},\mathcal{E}_{a_{1}})\otimes...\otimes\Ext^{k_{p-1}}(\mathcal{E}_{a_{p-1}},\mathcal{E}_{a_{p}})\\
&\qquad \qquad \qquad \qquad \qquad  \otimes\Ext^{k_{p}}(\mathcal{E}_{a_{p}},\CO_{y})
\end{aligned}
\end{equation}
converging to $\Ext^{p+q}(\iota^{\ast}\CO_{y},\iota^{\ast}\CO_{y})$.
For $ i\in \{1,\cdots,p-1\}$, Ext group $\Ext^{k_{i}}(\mathcal{E}_{a_{i}},\mathcal{E}_{a_{i+1}})$ is non-zero 
if and only if $k_{i}=2$ and at most one of $\mathcal{E}_{a_{i}}$ and $\mathcal{E}_{a_{i+1}}$ takes the form $\CO(-D_{l})$ for some $1\leq l\leq 11$.
Hence,
we derive $E_{1}^{-1-p,q}\neq 0$ if and only if  $0\leq p\leq 3$ and $q=-2(-1-p)$.
To determine all non-zero terms at $E_{1}$-page,
we define $\varphi(\mathcal{E}_{a_{0}},\cdots,\mathcal{E}_{a_{m}})$ to be the dimension of the complex vector space
$$
\Ext^{2}(\CO_{y},\mathcal{E}_{a_{0}})\otimes\Ext^{2}(\mathcal{E}_{a_{0}},\mathcal{E}_{a_{1}})\otimes \cdots \otimes\Ext^{2}(\mathcal{E}_{a_{m-1}},\mathcal{E}_{a_{m}})\otimes \Hom(\mathcal{E}_{a_{m}},\CO_{y}).
$$
Then, we deduce that $\dim E_{1}^{-1,2}=\sum_{i=1}^{14}\varphi(\mathcal{E}_{i})=14$ and 
\begin{align*}
 \dim E_{1}^{-2,4} &=
11\varphi(\CO(-D_{i}),\CO)+11\varphi(\CO(-F),\CO(-D_{i}))+\varphi(\CO(-2F),\CO)\\
&+
11\varphi(\CO(-2F),\CO(-D_{i}))+\varphi(\CO(-F),\CO)+\varphi(\CO(-2F),\CO(-F))\\&=
11\cdot 1+11\cdot 2+6+11\cdot 5+3+3\\&=
100.
\end{align*}
Similarly, by direct computations, we have
$$
\dim E_{1}^{-3,6}=152\,\;
\textrm{and}\; \dim E_{1}^{-4,8}=66.
$$
All differentials vanish except for
$d_{1}:E_{1}^{-1,2}\rightarrow E_{1}^{0,2}$, which is the sum of the evaluation maps
\begin{equation}\label{evaluation-map-surj}
\bigoplus_{i=1}^{14}\Ext^{2}(\CO_{y},\mathcal{E}_{i})\otimes\Hom(\mathcal{E}_{i},\CO_{y})\to\Ext^{2}(\CO_{y},\CO_{y})=\CN.
\end{equation}
Any element of $\Hom(\mathcal{E}_{i},\CO_{y})$ corresponds to an element of the stalk of $\mathcal{E}_{i}$ at $y$.
Thus, the evaluation map is 
$$\Ext^{2}(\CO_{y},\mathcal{E}_{i})\to\Ext^{2}(\CO_{y},\CO_{y}). $$ 
By Serre duality, we have the dual map 
$$\Hom(\CO_{y},\CO_{y})\rightarrow\Hom(\mathcal{E}_{i},\CO_{y}), $$
which is non-zero as it corresponds to scaling the chosen section of $(\mathcal{E}_{i})_{y}$.
Thus, the dual map is injective.
It follows that the map \eqref{evaluation-map-surj} is surjective, $\dim E_{2}^{-1,2}=13$ and  $\dim E_{2}^{0,2}=0$.
Hence, we have $
\Ext^{1}(\iota^{\ast}\CO_{y},\iota^{\ast}\CO_{y})=E_{1}^{0,1}\oplus E_{2}^{-1,2}=\CN^{15}.
$
In conclusion, we get $E_{2}^{-1,2}=\CN^{13}$, $E_{2}^{0,2}=0$ and all other terms $E_{2}^{\bullet,\bullet}=E_{1}^{\bullet,\bullet}$ at $E_{2}$-page.
Therefore, the spectral sequence \eqref{skyscraper-equ-1} degenerates at $E_{2}$-page and the proposition holds.
\end{proof}

\begin{rem}
The above proposition also implies that the phantom subcategory $\A_{Y}$ is non-trivial; namely, the exceptional collection \eqref{num-exceptional-collection} is non-full.
\end{rem}


\appendix

\section{The second Hochschild cohomology of $\A_{Y}$}\label{Appendix-HH-AY}

In this appendix,
we will obtain the dimension of the second Hochschild cohomology of phantom subcategory $\A_{Y}$ in Theorem \ref{main-thm}.

Let $\pi: Y \rightarrow \PB^{2}$ be the blow-up of the complex projective plane $\PB^{2}$ in $n$ general closed points $p_{i}$ ($n\geq 10$), where $1\leq i\leq n$.
We use $E_{i}:=\pi^{-1}(p_{i})\subset Y $ to denote the $(-1)$-curve over the points $p_{i}$. First we calculate the Hochschild cohomology of $Y$. 
We need the cohomology of wedge product of tangent sheaf of $Y$.

\begin{lem}\label{coh-tangent}
Let $\CT_{Y}$ be the tangent sheaf of $Y$. 
Then, we have
$$
\HO^{\bullet}(Y,\wedge^{j} \CT_{Y}) 
= 
\left\{ 
\begin{array}{cc}
\CN[0], & j=0,\\
\CN^{2n-8}[-1], & j=1,\\
\CN^{n-10}[-1], & j=2.
\end{array}
\right.
$$
\end{lem}

\begin{proof}
For $j=0$, then $\HO^{\bullet}(Y,\wedge^{j} \CT_{Y})=\HO^{\bullet}(Y,\CO_{Y})=\CN[0]$.
For $j=1$, the global sections of $\mathcal{T}_Y$ correspond to those vector fields on ${\PB^2}$ which vanishes at the $n$ general points. 
Therefore 
$$
h^{0}(T_{Y}):=\dim \HO^{0}(Y,\CT_{Y})=\max\{0,8-n\}=0.
$$  
Next, we calculate $\HO^{i}(Y,\CT_{Y})$ for $i>0$. Consider the following short exact sequence for the tangent sheaves 
$$
\xymatrix@C=0.5cm{
0\ar[r] & \CT_{Y} \ar[r] & \pi^*\CT_{\PB^2} \ar[r] & \bigoplus_i \CO_{E_i}(1) \ar[r] & 0}.
$$
Since $R\pi_{\ast}\CO_{Y}\cong \CO_{\PB^{2}}$, 
by projection formula, 
we have $
\HO^i(Y,\pi^*\CT_{\PB^2}) \cong \HO^i(\PB^2, \CT_{\PB^2}).
$
Note that $\HO^i(Y, \CO_{E_i}(1)) = \HO^i(\PB^1,\CO(1))$, then by the long exact sequence, we obtain that $\HO^{1}(Y,\CT_{Y})=\mathbb{C}^{2n-8}$ and the other cohomologies are zero (also see  \cite[Lemma 3.4]{KKL+26}).

For $j=2$, note that $\HO^{i}(Y,\wedge^{2}\CT_{Y}) = \HO^{i}(Y,\CO(-K_Y))$. 
Denote the scheme of union of $n$ blow-up points on $\PB^2$ by $Z$ and $I_{Z}$ the ideal sheaf of $Z$ in $\PB^{2}$.
Since $R\pi_{\ast}\left(\CO(-\sum_{i=1}^{11}E_{i})\right)\cong I_{Z}$, by projection formula, 
we have 
$$
\HO^i(Y,\CO(-K_Y)) \cong \HO^i(\PB^2, I_Z(3)).
$$
Taking cohomology of the short exact sequence
$$
\xymatrix@C=0.5cm{
0 \ar[r] & I_Z(3) \ar[r] & \CO_{\PB^2}(3) \ar[r] & \CO_Z \ar[r] & 0,}
$$
we obtain a long exact sequence
$$
\xymatrix@C=0.5cm{
0\ar[r] & \HO^{0}(I_Z(3)) \ar[r] & \HO^0(\CO_{\PB^2}(3)) \ar[r] & \HO^{0}(\CO_Z) \ar[r] & \HO^1(I_Z(3)) \ar[r] & 0.}
$$
The $\HO^{0}(I_Z(3))$ parametrizes the linear space of degree $3$ polynomials with $3$ variables which vanish at $n$ general points. 
Therefore $ h^{0}(I_{Z}(3))=\mathrm{max}\{0,10-n\}=0$ and 
$$
h^{1}(I_{Z}(3))= h^{0}(I_{Z}(3))-h^{0}(\CO_{\PB^2}(3))+h^0(\CO_Z)=n-10.
$$
This completes the proof.
\end{proof}

Recall the Hochschild--Kostant--Rosenberg (HKR) theorem:
$$
\HH^i(Y) = \bigoplus_{p+q=i} \HO^q(Y,\wedge^p\mathcal{T}_Y),
$$
it follows from Lemma \ref{coh-tangent} that we have the following result:
\begin{cor}\label{Hoch-cohomology-Y}
The Hochschild cohomology of $Y$ is given by 
$$
\HH^{0}(Y)=\CN,\,\HH^{1}(Y)=0,\,\HH^{2}(Y)=\CN^{2n-8},\,\HH^{3}(Y)=\CN^{n-10}$$
and $\HH^{i}(Y)=0 \textrm{ for } i\geq 4$.
\end{cor}

Now suppose $Y$ is the blow-up of $\PB^{2}$ in $11$ general closed points. 
We use
$$
\mathcal{B}:=\langle\CO_{Y},\CO_{Y}(D_{1}), \cdots,\CO_{Y}(D_{11}),\CO_{Y}(F),\CO_{Y}(2F) \rangle\subset \DC(Y)
$$
to denote the admissible subcategory generated by the sequence \eqref{num-exceptional-collection}.
We mainly calculate 
the third normal Hochschild cohomology of $\mathcal{B}$ in $\DC(Y)$ and then obtain the dimension of the second Hochschild cohomology $\HH^{2}(\A_{Y})$.

\begin{prop}\label{normal-Hoch-cohomology-Y}
The third normal Hochschild cohomology $\NHH^{3}(\mathcal{B},Y)=\CN^{14}$ and the second Hochschild cohomology $ \HH^{2}(\mathcal{A}_{Y})= \CN^{27}$.
\end{prop}

\begin{proof}
By \cite[Proposition 3.7]{Kuz15}, there exists a spectral sequence converging to the normal Hochschild cohomology
\begin{equation}\label{NHH--spectral-1}
E_1^{-p,q} \Rightarrow \NHH^{q-p}(\mathcal{B}, Y),
\end{equation}
where $E_{1}^{-p,q}$ is given by 
\begin{equation*}\label{NHH-spectral-seq}
\bigoplus_{\substack{0 \leq a_0 < \cdots < a_p \leq 13 \\ k_0 + \cdots + k_p = q}} \Ext^{k_0}(L_{a_0}, L_{a_1}) \otimes \cdots \otimes 
\Ext^{k_{p-1}}(L_{a_{p-1}}, L_{a_p}) \otimes 
 \Ext^{k_p}(L_{a_p}, \mathrm{S}^{-1}(L_{a_0})),
\end{equation*}
$\mathrm{S}(-):=-\otimes \omega_{Y}[2]$ is the Serre functor of $\DC(Y)$, $L_{0}:=\CO_{Y}$, $L_{i}:=\CO_{Y}(D_{i-1})$($1\leq i\leq 11$), $L_{12}:=\CO_{Y}(F)$ and  $L_{13}:=\CO_{Y}(2F)$.
By direct computations, we derive that
$$\RHom(L_{a_{p}},\mathrm{S}^{-1}(L_{a_{0}}))=\CN^{-\chi(\mathrm{S}^{-1}(L_{a_{0}})\otimes L_{a_{p}}^{\vee})}[-3]$$
and in particular, the term $\Ext^{3}(L_{a_{0}},\mathrm{S}^{-1}(L_{a_{0}}))=\CN$.
By Lemma \ref{strong-generator-lem-2}, we know that the terms $\Ext^{k_i}(L_{a_i}, L_{a_{i+1}})$ have nontrivial values only for $k_{i}=2$.
Hence, $E_{1}^{-p,q}$ is nontrivial only if $2p+3=q$.
If $p-q=3$ and $2p+3=q$, then $p=0$ and $q=3$.
It follows that $E_{1}^{0,3}=\bigoplus_{a_{0}=0}^{13}\Ext^{3}(L_{a_{0}},\mathrm{S}^{-1}(L_{a_{0}}))=\CN^{14}$.
Therefore, by the spectral sequence \eqref{NHH--spectral-1}, the normal Hochschild cohomology $\NHH^{3}(\mathcal{B},Y)=\CN^{14}$.

By Lemma \ref{excep-coll3}, we obtain $\NHH^{2}(\mathcal{B},Y)=0$.
According to \cite[Theorem 3.3]{Kuz15}, we have the following exact sequence:
$$
\xymatrix@C=0.5cm{
0 \ar[r] & \HH^{2}(Y) \ar[r] & \HH^{2}(\mathcal{A}_{Y}) \ar[r] & \mathrm{NHH}^{3}(\mathcal{B},Y) \ar[r]^{\delta} & \HH^{3}(Y).
}
$$
By Corollary \ref{Hoch-cohomology-Y} above, we have $\HH^{2}(Y)=\CN^{14}$ and $\HH^{3}(Y)=\CN$.
By Lemma \ref{non-trivial map}, the map $\delta$ is non-trivial and thus surjective. 
Using the above exact sequence, we conclude $ \dim \HH^{2}(\A_{Y})=27$.
\end{proof}

\begin{lem}\label{non-trivial map}
 The map $\delta$ is non-trivial.  
\end{lem}

\begin{proof}
Let $\mathscr{D}_{Y}$ be the DG enhancement of $\DC(Y)$ constructed in \cite[Section 2.4]{Kuz15}, and let $\mathscr{B}\subset \mathscr{D}$ be the natural induced DG enhancement of $\mathcal{B}$. 
Then there is an exact triangle 
$$
\xymatrix@C=0.5cm{
Q \ar[r] & \Delta_{\ast}\CO_{Y} \ar[r] & P,
}
$$
where $\Delta: Y\hookrightarrow Y\times Y$ is the diagonal, $Q\in \mathcal{B}\boxtimes \DC(Y)$ and
$P\in \mathcal{A}_{Y}\boxtimes \DC(Y)$.
By \cite[Proposition 2.6]{Kuz15},  
$Q\simeq \mu(\mathscr{D}\otimes_{\mathscr{B}}^{\mathbf{L}}\mathscr{D}),
$
where $\mu$ is the equivalence from $\mathscr{D}$-$\mathscr{D}$-bimodules to Fourier--Mukai kernels on $Y\times Y$ obtained in \cite[Page 12]{Kuz15}.
The kernel morphism
$
Q\longrightarrow \Delta_{\ast}\CO_{Y}
$
is induced by the DG composition morphism
$
\mathscr{D}\otimes_{\mathscr{B}}^{\mathbf{L}}\mathscr{D}\longrightarrow \mathscr{D}.
$
Moreover, the map
$\delta$
is induced by this kernel morphism $Q\to \Delta_*\CO_{Y}$.
In the direct sum of $\mathscr{D}\otimes_\mathscr{B}^{\mathbf L}\mathscr{D}$ in \cite[Page 15]{Kuz15}, the $p=0$ summand corresponding to $L_{i}$ is
$
\mathscr{D}(-,L_{i})\otimes \mathscr{D}(L_{i},-), 0 \leq i\leq 13.
$
Under the functor $\mu$, this bimodule corresponds, up to the convention for the two factors, to the kernel
$
L_{i}\boxtimes L_{i}^\vee.
$
The composition morphism on this summand is the usual composition pairing
$$
\mathscr{D}(-,L_{i})\otimes \mathscr{D}(L_{i},-)
\longrightarrow
\mathscr{D}(-,-),
$$
and therefore under $\mu$ it becomes a non-trivial  morphism
$
e_{i}:
L_{i}\boxtimes L_{i}^\vee
\longrightarrow
\Delta_*\mathcal O_Y.
$
By adjunction, we obtain $\Hom(L_{i}\boxtimes L_{i}^{\vee},\Delta_{\ast}\CO_{Y})\cong \Hom(L_{i}\otimes L_{i}^{\vee},\CO_{Y})=\CN$.
Thus, we derive
$e_{i}=\Delta_{\ast}(\mathrm{ev}_{L_{i}})\circ \eta$ up to a nonzero scalar,
where $\eta$ is the counit of $L_{i}\boxtimes L_{i}^{\vee}$ under $\Delta$ and $\mathrm{ev}_{L_{i}}$ is the evaluation morphism $L_{i}\otimes L_{i}^{\vee}\rightarrow \CO_{Y}$.

Since $
\Delta^{!}(L_{i}\boxtimes L_{i}^\vee)
\cong
\omega_{Y}^{-1}[-2]
$
and 
$
\Delta^{!}\Delta_{\ast} \CO_{Y}
\cong 
\mathcal \CO_{Y}\oplus T_{Y}[-1]\oplus\omega_{Y}^{-1}[-2].
$
The morphism 
$
\Delta^{!}e_{i}=\Delta^{!}(\Delta_{\ast}(\mathrm{ev}_{L_{i}})\circ \eta)
$ induces, up to a nonzero scalar factor determined by the evaluation morphism $\mathrm{ev}_{L_{i}}$, the identity map from $\omega_{Y}^{\vee}[-2]$ to itself.
Taking cohomology, we get 
$\delta=\bigoplus_{i=0}^{13}H^{3}(\Delta^{!}e_{i})\neq 0$.
\end{proof}

\section{No phantoms on weak del Pezzo surfaces}
\label{No-Ph-Wdel}

In this appendix, we verify Conjecture 1.3 of Borisov--Kemboi \cite{BK25} for the following surfaces:

\begin{prop}\label{Wdl-no-phantoms}
Let $S$ be a smooth projective surface with an effective smooth anti-canonical divisor $E$. 
Then $\DC(S)$ has no phantoms.
\end{prop}

\begin{proof}
Suppose that $\mathcal{C}\subset \DC(S)$ is a non-trivial phantom subcategory.
Let $Z:=\supp(\mathcal{C})=\bigcup_{F\in \mathcal{C}} \supp(F)$ be the  support of $\mathcal{C}$, where 
$\supp(F)$ is the support of the object $F$. Then, by \cite[Lemma 3.6]{Pir26}, $Z\subset S$ is a closed subset.
Note that there is a semi-orthogonal decomposition
$$
\DC(S)=
\langle 
\mathcal{C}^{\perp},\mathcal{C}
\rangle.
$$
Since $S$ is a rational surface and $E$ is an elliptic curve, by the same arguments as Step $1$ and Step $2$ in the proof of \cite[Theorem 3.1]{BK25},
we obtain that the skyscraper sheaf 
$\CO_{p}\in \mathcal{C}^{\perp}$ for every closed point $p\in E$.
As a result, 
for every object $F\in \mathcal{C}$, we have $\RHom(F,\CO_{p})=0$. 
This means that $p\not\in \supp(F)$ for all points $p\in E$. 
Thus, we have $Z\cap E =\emptyset$.
Since we assume $\mathcal{C}$ is non-trivial, by \cite[Lemma 6.16]{Pir23}, $\mathcal{C}$ must have $1$-dimensional support. 
Then, according to \cite[Theorem 1.1]{Pir26}, there exists an irreducible  component $C\subset Z$ such that 
$$
K_{S}\cdot C<0.
$$
Since $E\in |-K_{S}|$, so $E\cdot C=(-K_{S})\cdot C>0$.
However, we have $E\cdot C=0$ since $Z\cap E=\emptyset$, a contradiction.
\end{proof}

In particular, we have:

\begin{cor}\label{important-special-case}
There are no phantom admissible subcategories on  
weak del Pezzo surfaces, on the blow-up of a weak del Pezzo surface along any finite points in general position on an anti-canonical divisor, and on the the blow-up of the second Hirzebruch surface $\mathbb{F}_{2}$ at $d\leq 8$ points in general position.  
\end{cor}

\begin{proof}
Note that all the rational surfaces in corollary have an effective smooth anti-canonical divisor. 
Thus, it follows immediately from Proposition \ref{Wdl-no-phantoms}. 
\end{proof}

\begin{rem}  
(1) Let $S$ be the blow-up of $\PB^{2}$ at finite points on a smooth cubic curve $C$. 
Then $S$ has no phantoms. 
In fact, let $\tilde{C}$ be the strict transform of $C$.
Then $E:=\tilde{C}\in|-K_{S}|$ is a smooth curve.

(2) Suppose that $S$ is a weak del Pezzo surface but not del Pezzo.
Then there exists a $(-2)$-curve $C\subset S$.
Thus, $E\cdot C=(-K_{S})\cdot C=0$
and $\CO_{S}(C)|_{E}\cong \CO_{E}$.
Consequently, the restriction map $\mathrm{Pic}(S) \rightarrow \mathrm{Pic}(E)$ is not injective.
Therefore, Proposition \ref{Wdl-no-phantoms} is a non-trivial generalization of both \cite[Theorem 6.35]{Pir23} and \cite[Theorem 1.1]{BK25}. 
\end{rem}

\begin{rem}
Recall that a weak del Pezzo surface is isomorphic to the $0$-th Hirzebruch surface $\mathbb{F}_{0}$, the second Hirzebruch surface $\mathbb{F}_{2}$ or to the blow-up of a bubble cycle on $\PB^{2}$ that consists of $\leq 8$ points.
By Corollary \ref{important-special-case}, \cite[Conjecture 4.10]{KKL+26} holds for the case of $\mathbb{F}_{2}$.
It is still an open problem for the case of the Hirzebruch surface $\mathbb{F}_{n}$ with $n\geq 3$.   
\end{rem}


\end{document}